\documentclass[11pt,reqno,draft]{amsart}

\usepackage[PostScript=dvips,dpi=600,nohug,small]{diagrams}
\usepackage{amssymb}
\usepackage{mathrsfs}
\usepackage{cite}
\usepackage{amsfonts}
\newcommand{\rmv}[1]{}

\def\wh{\widehat}

\def\MM{\ensuremath{{\mathscr M}}}
\def\pv#1{\ensuremath{{\bf#1}}}

\def\inv{^{-1}}
\def\p{\varphi}

\def\pinv{{\p \inv}}
\def\J{\mathrel{{\mathscr J}}} 

\def\L{\mathrel{{\mathscr L}}} 
\def\H{\mathrel{{\mathscr H}}} 

\def\e<{\leq _{E}}

\def\ov#1{\ensuremath{\overline {#1}}}

\def\til#1{\ensuremath{\widetilde {#1}}}

\def\malce{\mathbin{\hbox{$\bigcirc$\rlap{\kern-8.3pt\raise0,50pt\hbox{$\mathtt{m}$}}}}}

\def\FP#1#2{\ensuremath{\widehat{F}_{\mathbf #1}(#2)}}
\def\ovFP#1#2{\ensuremath{\widehat{F}_{\overline{\mathbf #1}}(#2)}}

\def\1sk{^{(1)}}

\def\to{\rightarrow}

\def\Thmname{Theorem}
\def\Propname{Proposition}
\def\Lemmaname{Lemma}
\def\Definitionname{Definition}
\newtheorem{Thm}{\Thmname}
\newtheorem{Prop}[Thm]{\Propname}
\newtheorem{Lemma}[Thm]{\Lemmaname}

\newtheorem{Cor}[Thm]{Corollary}

\newtheorem{Conjecture}[Thm]{Conjecture}

\numberwithin{equation}{section}

\title[On Free Profinite Subgroups of Free Profinite Monoids]{On Free Profinite Subgroups of Free Profinite Monoids}

\author{Benjamin Steinberg}
\address{School of Mathematics and Statistics\\
Carleton University \\
1125 Colonel By Drive\\
Ottawa, Ontario  K1S 5B6 \\
Canada}
\thanks{The author was supported in part by NSERC}
\email{bsteinbg@math.carleton.ca}
\date{\today}


\begin{document}
\begin{abstract}
We answer a question of Margolis from 1997 by establishing that the
maximal subgroup of the minimal ideal of a finitely generated free
profinite monoid is a free profinite group.  More generally if $\pv
H$ is variety of finite groups closed under extension and containing
$\mathbb Z/p\mathbb Z$ for infinitely may primes $p$, the
corresponding result holds for  free pro-$\ov{\pv H}$ monoids.
\end{abstract}
\maketitle

\section{Introduction}
Margolis asked in 1997 whether the maximal subgroup of the minimal
ideal of a finitely generated free profinite monoid is a free
profinite group;  this question first appeared in print (to the best
of our knowledge) in our paper with Rhodes~\cite{RhodesStein}. The
question was prompted by the discovery of free profinite subgroups
by Almeida and Volkov~\cite{AlmeidaVolkovfirst}, who subsequently
characterized those free profinite subgroups which are
retracts~\cite{AlmeidaVolkov}. Recently, Almeida has shown that not
all maximal subgroups of finitely generated free profinite monoids
are free profinite groups~\cite{Jorgesubgroup}, although he has provided
a large class of examples that are free profinite.  His fascinating technique
involves a correspondence between symbolic dynamical systems in
$X^{\omega}$ and certain $\J$-classes of the free profinite monoid
$\wh{X^*}$. In particular, his methods apply best to \emph{maximal}
infinite $\J$-classes, which correspond to \emph{minimal} dynamical
systems. The minimal ideal corresponds to the full shift
$X^{\omega}$ and so Almeida's approach does not yet apply to
studying this maximal subgroup.

The author and Rhodes recently established that closed subgroups of
free profinite monoids are projective profinite
groups~\cite{projective}.  This answered a question raised by
several people including Almeida, Margolis, Lubotzky and the
author~\cite{RhodesStein}.  Projectivity is a necessary, but far
from sufficient condition, for freeness~\cite{RZbook}.  In this
paper we answer Margolis's question in the affirmative.  We also
prove the analogous result relative to certain varieties of finite
groups. Recall that if $\pv H$ is a variety of finite groups, that
is a class of finite groups closed under taking direct products,
subgroups and quotient groups, then the class $\ov{\pv H}$ of
monoids whose subgroups belong to $\pv H$ is a variety of finite
monoids. Our main result is then:

\begin{Thm}\label{Theorem1}
Let $\pv H$ be a variety of finite groups closed under extension,
which contains $\mathbb Z/p\mathbb Z$ for infinitely may primes $p$.
Then the maximal subgroup of the minimal ideal of a finitely
generated (but not procyclic) free pro-{\ov {\pv H}} monoid is a
free pro-\pv H group of countable rank.
\end{Thm}

If \pv V is a variety of finite monoids, then $\FP V X$ denotes the
free pro-\pv V monoid generated by $X$.  The natural projection
$\pi:\ovFP H X\to \FP H X$ restricts to an epimorphism on the
maximal subgroup $G$ of the minimal ideal of $\ovFP H
X$~\cite{AlmeidaVolkov,RhodesStein}.  Our second result describes
the kernel of the epimorphism $G\twoheadrightarrow \FP H X$.

\begin{Thm}\label{Theorem2}
Let $\pv H$ be a variety of finite groups closed under extension,
containing $\mathbb Z/p\mathbb Z$ for infinitely may primes $p$, and
let $X$ be a finite set of cardinality at least two. Let
$\p:G\twoheadrightarrow \FP H X$ be the canonical epimorphism, where
$G$ is the maximal subgroup of the minimal ideal of $\ovFP H X$.
Then $\ker \p$ is a free pro-\pv H group of countable rank.
\end{Thm}

It seems likely that our results hold for any non-trivial
extension-closed variety of finite groups.  The hypothesis on primes
is entirely of a technical nature and should not really be essential. For
example, since projective pro-$p$ groups are free
pro-$p$~\cite{RZbook}, Theorems~\ref{Theorem1} and~\ref{Theorem2}
are valid for \pv H the variety of finite $p$-groups.   We further propose
the following conjecture.

\begin{Conjecture}
Under the hypotheses of Theorem~\ref{Theorem1} the maximal subgroup
of the closed subsemigroup generated by the idempotents of the
minimal ideal of a finitely generated (non-procyclic) free pro-{\ov
{\pv H}} monoid is a free pro-\pv H group of countable rank.
\end{Conjecture}

In fact, we suspect a slight variation of the construction used to
prove Theorem~\ref{Theorem1} already suffices to prove the
conjecture, the remaining issues being purely technical.  The proof
of Theorem~\ref{Theorem1} relies on a criterion for freeness, due to
Iwasawa~\cite{Iwasawa}, and extensive usage of wreath products. In
spirit the proof draws from the following sources: our previous work
with Rhodes~\cite{projective}, the synthesis
theorem~\cite{synthesis} and the classical construction embedding
any countable group as a maximal subgroup of a two-generated monoid
consisting of a cyclic group of units and a completely simple
minimal ideal.

\section{Minimal ideals}\label{minimal}
In this section we collect a number of standard facts concerning
minimal ideals in finite and profinite semigroups, which can be
found, for instance, in~\cite{CP,Arbib,qtheor}. If $S$ is a
semigroup, then $E(S)$ denotes the set of idempotents of $S$. For an
idempotent $f\in E(S)$, the group of units $G_f$ of the monoid $fSf$
is called the maximal subgroup of $S$ at $f$.

The first fact is that every profinite monoid $M$ has a unique
minimal ideal $I$.  It is necessarily closed and if $x\in I$, then
$I=MxM$.  Since every compact semigroup contains an idempotent, it
follows that $I$ contains an idempotent $e$.  Since compact
semigroups are stable~\cite{qtheor}, Green-Rees structure
theory~\cite{CP,Arbib,qtheor} implies that the maximal subgroup
$G_e$ is $eIe$ and furthermore is a closed subgroup (and hence a
profinite group), which is independent of the choice of $e$ up to
isomorphism.

\begin{Prop}\label{minontominimal}
Let $\p:S\twoheadrightarrow  T$ be a continuous onto homomorphism of
profinite monoids.  Let $I$ be the minimal ideal of $S$ and $J$ be
the minimal ideal of $T$.  Then $\p(I)=J$ and moreover, if $e\in
E(I)$, then $\p(G_e)$ is the maximal subgroup of $J$ at $\p(e)$.
\end{Prop}
\begin{proof}
Clearly $\pinv(J)$ is an ideal
of $S$ so $I\subseteq \pinv(J)$, i.e.\ $\p(I)\subseteq J$. On the
other hand $\p(I)$ is an ideal of $T$ since $\p$ is onto.  Thus
$\p(I) = J$ by minimality.  Now $\p(G_e) = \p(eIe) = \p(e) \p(I)
\p(e) = \p(e) J\p(e) = G_{\p(e)}$, completing the proof.
\end{proof}

In particular, every profinite group image of a profinite monoid $M$
is an image of the maximal subgroup of its minimal ideal.

The minimal ideal $I$ of a finite monoid $M$ is a simple semigroup,
and hence isomorphic to a Rees matrix semigroup $\MM(G,A,B,C)$ where
$C:B\times A\to G$ is the sandwich matrix~\cite{CP,Arbib,qtheor}.
Let $a_0\in A$ and $b_0\in B$.  Then without loss of generality we
may assume that each entry of row $b_0$ and of column $a_0$ is the
identity of $G$~\cite{Arbib,qtheor}.  We can identify $G$ with the
$\H$-class $a_0\times G\times b_0$.

Recall that $B$ can be identified with the $\L$-classes of $I$.
There is a natural action of $M$ on the right of $B$ since $\L$ is a
right congruence.  Let $(B,\mathsf{RLM}_I(M))$ be the associated
faithful transformation semigroup.  Notice that each element of $I$
acts on $B$ as a constant map and that all constant maps on $B$
arise from elements of $I$. The Sch\"utzenberger representation
gives a wreath product representation $M\to G\wr
(B,\mathsf{RLM}_I(M))$~\cite{CP,Arbib,qtheor}. An element
$s=(a,g,b)\in I$ is sent to the element $(f_s,\ov b)$ where $b'f_s =
b'Cag$ and $\ov b$ is the constant map to $b$. In particular, if
$s=(a_0,g,b_0)$ is an element of our maximal subgroup, then $b'f_s =
g$ all $b'\in B$.  Consequently, the Sch\"utzenberger representation
is faithful on the maximal subgroup $G$.

We recall that the finite simple  semigroups form a variety of
finite semigroups denoted $\pv {CS}$.  It is well
known~\cite{Almeida:book} that $\pv {CS} = \pv G\ast \pv {RZ}$, that
is, it consists precisely of divisors of wreath products of finite
groups and right zero semigroups. In fact, we shall need the
following more explicit lemma.

\begin{Lemma}\label{structureofwreath}
Let $S = G\wr (B,\ov B)$ where $G$ is a finite group and $\ov B$ is
the semigroup of constant maps on the set $B$.  Then $S$ is simple
and the maximal subgroup of $S$ is isomorphic to $G$.  More
precisely, if $e=(f,\ov b)$ is an idempotent then the map
$\psi:eSe\to G$ given by $\psi(f',\ov b)= bf'$ is an isomorphism.
\end{Lemma}
\begin{proof}
We have already observed that $S$ is simple (this can also be
verified by direct computation). Let $e=(f,\ov b)$ be an idempotent
of $S$. We must show $\psi$ defined as above is an isomorphism.
  First we verify
$\psi$ is a homomorphism. Indeed, $(f',\ov b)(f'',\ov b) =
(f'{}^{\ov b}\!{f''},\ov b)$ and $b(f'{}^{\ov b}\!{f''}) = bf'bf''$.
In particular, we have $1=\psi(e)=bf$.

To see $\psi$ is injective, note that $(f',\ov b)\in G_e$ implies
$(f',\ov b) = (f,\ov b)(f',\ov b) = (f{}^{\ov b}\!{f'},\ov b)$ and
so $b'f' = b'fbf'$ all $b'\in B$.  Thus $f'$ is determined by
$bf'=\psi(f',\ov b)$ and so $\psi$ is injective. Finally to verify
$\psi$ is onto, let $g\in G$ and consider $(f',\ov b)=e(\ov g,\ov
b)e$ where $\ov g$ is the constant map $B\to G$ taking all of $B$ to
$g$.  Then $bf' = (bf)(b\ov g)(bf) = g$ since $bf=1$.  Thus
$\psi(f',\ov b) =g$, establishing $\psi$ is onto.
\end{proof}

\section{The proofs of Theorems~\ref{Theorem1} and~\ref{Theorem2}}
In this section we prove Theorems~\ref{Theorem1} and~\ref{Theorem2}
modulo two technical lemmas.  Fix a variety of finite groups \pv H
closed under extension and containing $\mathbb Z/p\mathbb Z$ for
infinitely many primes $p$. Denote by $\ov {\pv H}$ the variety of
finite monoids whose subgroups belong to \pv H.

\subsubsection*{Proof of Theorem~\ref{Theorem1}}
 Let $X=\{x_1,\ldots,x_n\}$ be a finite set of cardinality at least
two. Denote by $I$ the minimal ideal of $\ovFP H X$. Choose an
idempotent $e\in I$. Recall that if $x$ is an element of a profinite
semigroup, then $x^{\omega} = \lim x^{n!}$ is the unique idempotent
in the closed subsemigroup generated by $x$. Without loss of
generality we may assume $x_1^{\omega}e=e=ex_2^{\omega}$; if not
replace $e$ with $(x_1^{\omega}ex_2^{\omega})^{\omega}$.
 Let $G_e$ be the maximal
subgroup at $e$. Our goal is to show  $G_e$ is free pro-\pv H on a
countable set of generators converging to the identity (that is,
free of countable rank).  Recall that a subset $Y$ of a profinite
group $G$ converges to the identity if each neighbourhood of the
identity contains all but finitely many elements of $Y$.  A pro-\pv
H group $F$ is free pro-\pv H on a subset $Y$ converging to the
identity if given any map $\tau:Y\to H$ with $H$ pro-\pv H and
$\tau(Y)$ converging to the identity, there is a unique extension of
$\tau$ to $F$.  Any free pro-\pv H group on a profinite space has a
basis converging to the identity.  See~\cite{RZbook} for details.

It is well known $\ovFP H X$ is
metrizable~\cite{Almeida:book,qtheor}, and hence so is $G_e$.
Thus the identity $e$ of $G_e$ has a countable basis of
neighbourhoods.  We shall use a well-known criterion, going back to
Iwasawa~\cite{Iwasawa}, to establish $G_e$ is free pro-\pv H of
countable rank. An
\emph{embedding problem} for $G_e$ is a diagram
\begin{equation}\label{embeddingproblem}
\begin{diagram}
 &                & G_e\\
 &                 &\dOnto_{\p}\\
H&\rOnto^{\alpha}&K
\end{diagram}
\end{equation}
with $H\in \pv H$ and $\p,\alpha$ epimorphisms ($\p$ continuous).
A \emph{solution} to the embedding problem \eqref{embeddingproblem}
is a continuous \textbf{epimorphism} $\til\p:G_e\to H$ making the diagram
\[\begin{diagram}
\begin{diagram}
 &                & G_e\\
 &     \ldOnto^{\til\p}            &\dOnto_{\p}\\
H&\rOnto^{\alpha}&K
\end{diagram}
\end{diagram}\]
commute.  According to~\cite[Corollary 3.5.10]{RZbook} to prove
$G_e$ is free pro-\pv H of countable rank it suffices to show that
every embedding problem \eqref{embeddingproblem} for $G_e$ has a
solution. We proceed via a series of reductions on the types of
embedding problems we need to consider.

Since $G_e$ is a closed subgroup of $\ovFP H X$, there is a
continuous onto homomorphism $\p':\ovFP H X\twoheadrightarrow M'$
with $M'$ a finite monoid in $\ov{\pv H}$ such that $\ker
\p'|_{G_e}\leq \ker \p$. Setting $K' = \p'(G_e)$, let
$\rho:K'\twoheadrightarrow K$ be the canonical projection. Defining
$H'$ to be the pullback of $\alpha$ and $\rho$, that is $H' =
\{(h,k')\in H\times K'\mid \alpha(h)=\rho(k')\}$, yields a
commutative diagram
\begin{equation*}
\begin{diagram}
                   &                 &              &                & &G_e\\
                   &                 &              & &\ldOnto(3,2)^{\p'}\ldOnto(3,4)_{\p}  &\\
H'&\rOnto^{\alpha'} & K'         &  &              &\\
\dOnto^{\rho^*}             &                 & \dOnto_{\rho}&   &             & \\
H                  &\rOnto^{\alpha}  & K            &&&
\end{diagram}
\end{equation*}
where $\rho^*$ is the projection to $H$.  It is easily verified that
all the arrows in the diagram are epimorphisms.  So to solve our
original embedding problem, it suffices to solve the embedding
problem:
\[\begin{diagram}
\begin{diagram}
 &                & G_e\\
 &     &\dOnto_{\p'}\\
H'&\rOnto^{\alpha'}&K'
\end{diagram}
\end{diagram}\]
In other words, reverting back to our original notation, we may
assume  in the embedding problem \eqref{embeddingproblem} the
map $\p$ is the restriction of a continuous onto homomorphism
$\p:\ovFP H X\twoheadrightarrow M$ with $M\in \ov{\pv H}$.  Let $J$
be the minimal ideal of $M$; so $J=\p(I)$.  Then the right
Sch\"utzenberger representation~\cite{CP,qtheor,Schutzmonomial} of
$M$ on $J$ is faithful when restricted to $K$.  Possibly replacing
$M$ by its image under the Sch\"utzenberger representation, we may
assume that the right Sch\"utzenberger representation of $M$ on $J$
is faithful. Therefore, we may view $M$ as embedded in the wreath
product $K\wr (B,\mathsf{RLM}_J(M))$.   The existence of a solution
then follows from the following technical lemma, which is the
subject of Section~\ref{s:proofofmaintech}.

\begin{Lemma}\label{maintechnical}
Let $\p:\ovFP H X\twoheadrightarrow M$ be a continuous surjective
morphism, with $M$ finite, such that $\p(G_e) = K$ and the (right)
Sch\"utzenberger representation of $M$ on its minimal ideal $J$ is
faithful. Let $\alpha:H\twoheadrightarrow K$ be an epimorphism.
 Then there is an $X$-generated finite monoid $M'\in \ov{\pv H}$
 such that if $\eta:\ovFP H X\to M'$ is the continuous projection,
 then:
\begin{enumerate}
\item there is an isomorphism $\theta:G_{\eta(e)}\to H$ where
  $G_{\eta(e)}$ is the maximal subgroup at $\eta(e)$ of the minimal
  ideal of $M'$;
\item $\p$ factors through $\eta$ as $\rho\eta$ where
$\rho:M'\twoheadrightarrow M$ satisfies $\rho\theta\inv = \alpha$.
\end{enumerate}
\end{Lemma}

Assuming the lemma,  our
desired solution to the embedding problem \eqref{embeddingproblem} is
$\til{\p}=\theta \eta|_{G_e}:G_e\to H$.
Indeed, $\eta|_{G_e}$ is an epimorphism by
Proposition~\ref{minontominimal} and hence $\til{\p}$ is an
epimorphism. Moreover, $\alpha\til{\p}=\rho\theta\inv
\theta\eta|_{G_e} = \p|_{G_e}$ and so $\til{\p}$ is indeed a
solution to the embedding problem \eqref{embeddingproblem}. This
completes the proof of Theorem~\ref{Theorem1}.\qed

\subsubsection*{Proof of Theorem~\ref{Theorem2}} Let $\pi:\ovFP H X\to \FP H X$ be
the canonical projection; so $\p=\pi|_G$ where $G$ is the maximal
subgroup of the minimal ideal $I$ of $\ovFP H X$.  Let $N=\ker \p$.
We shall use a criterion due to Mel'nikov to prove that $N$ is free
pro-\pv H. We first need to recall the notion of
$S$-rank~\cite{RZbook}. If $S$ is a finite simple group and $G$ is a
profinite group, denote by $M_S(G)$ the intersection of all open
normal subgroups $N$ of $G$ such that $G/N\cong S$.  It is
known~\cite[Chapter 8.2]{RZbook} that $G/M_S(G)\cong \prod_A S$, a
direct product of copies of $S$ indexed by $A$. The cardinality
$r_S(G)$ of $A$ is called the \emph{$S$-rank} of $G$. One property
of $S$-rank that we  shall need is part of~\cite[Lemma
8.2.5]{RZbook}.

\begin{Lemma}\label{ctsimage}
Suppose $H$ is a continuous image of $G$, then $r_S(H)\leq r_S(G)$.
\end{Lemma}

Mel'nikov's criterion for freeness of a normal
subgroup~\cite[Theorem 8.6.8]{RZbook} is then:

\begin{Thm}[Mel'nikov]
Let $\pv H$ be a variety of finite groups closed under extension and
let $F$ be a free pro-\pv H group of countably infinite rank.  A
non-trivial closed normal subgroup $N$ of infinite index in $F$ is
free pro-\pv H (of countable rank) if and only if the $S$-rank
$r_S(N)$ is infinite for each finite simple group $S\in \pv H$.
\end{Thm}

In our context, since $G/N$ is a free profinite group of rank $|X|$,
clearly $N$ has infinite index.  So it suffices to show that $N$ has
infinite $S$-rank for all finite simple groups $S\in \pv H$. By
Lemma~\ref{ctsimage} it suffices to show $S^n$ is a continuous image
of $N$ for all $n\geq 1$ (as $r_S(S^n)=n$).  Notice that $\pi(E(I))
=1$ and so $\langle E(I)\rangle \cap G_e\leq N$ (one can in fact
show that $N$ is the closed normal subgroup generated by $\langle
E(I)\rangle \cap G_e$, but we shall not need this). The desired
result is then an immediate consequence of the following technical
lemma, which will be proved in Section~\ref{s:proofoftech2}.

\begin{Lemma}\label{technicallemma2}
Let \pv H be any variety of finite groups containing cyclic groups
of arbitrary cardinality and let $H\in \pv H$.  Then there is a
two-generated finite monoid $M\in \ov{\pv H}$ such that $H$ is the
maximal subgroup of the minimal ideal $J$ of $M$ and $J$ is
generated by idempotents.
\end{Lemma}

From Lemma~\ref{technicallemma2} we conclude every group in \pv H is
a continuous image of $N$, yielding Theorem~\ref{Theorem2}.  Indeed,
if $\psi:\ovFP H X\to M$ is the canonical surjection, then
$\psi(E(I))=E(\psi(I))$ and so, since $M$ is finite,
\[\psi(\ov{\langle E(I)\rangle})=\psi(\langle E(I)\rangle) = \langle
\psi(E(I))\rangle = \langle E(J)\rangle = J\] as $J = \psi(I)$.   By
Graham's theorem~\cite{Graham,qtheor} the idempotent-generated
subsemigroup of a finite simple semigroup is simple;  hence the
closed subsemigroup generated by the idempotents of a simple
profinite semigroup is simple.  Proposition~\ref{minontominimal}
then easily yields $\psi(N)$ is the maximal subgroup of $J$.

\section{The proof of Lemma~\ref{technicallemma2}}\label{s:proofoftech2}
We prove Lemma~\ref{technicallemma2} first since the proof is easier
and at the same time highlights many of the ideas that will be used
to prove Lemma~\ref{maintechnical}. The construction we use is a
variant on a classical construction.  Usually it is formulated in
terms of Rees matrix semigroups, but it will be more convenient for
us to use wreath products.  If $(Y,S)$ is a transformation monoid,
then $\ov{(Y,S)}$ denotes the augmented transformation monoid
obtained by adjoining to $S$ the constant maps on $Y$.  Set
$[n]=\{1,\ldots,n\}$.

Let $H\in\pv H$ and let $n\geq 2|H|-1$ be an integer such that
$\mathbb Z/n\mathbb Z$ belongs to $\pv H$.  Suppose $e$ is the
identity of $H$ and $H=\{e=h_1,\ldots,h_m\}$.  Let $C_n$ be the
cyclic group of order $n$ generated by the cyclic permutation $a=(1\
2\cdots n)$. Consider the following two elements of the wreath
product $H\wr \ov{([n],C_n)}$:
\begin{align*}
x &= (\ov e,a)\ \text{where}\ j\ov e =e,\ \text{all}\ j\in [n]\\
y &= (Y,\ov 1)\ \text{where}\ jY =\begin{cases} h_j & 1\leq j\leq m
\\ e & m<j\leq n.\end{cases}
\end{align*}

Let $M$ be the submonoid generated by $x$ and $y$.  First observe
$x$ is an invertible element of order $n$.  On the other hand $y$ is
an idempotent since $y^2 =(Y{}^{\ov 1}\!{Y},\ov 1)$ and $j(Y{}^{\ov
1}\!{Y}) = (jY)(1Y)=(jY)h_1=(jY)e=jY$.  A routine application of
Lemma~\ref{structureofwreath} yields the minimal ideal $J$ of $M$ is
$M\cap H\wr ([n],\ov{[n]})$.  Indeed, $M\cap H\wr ([n],\ov{[n]})$
contains $y$ and is a subsemigroup of the simple semigroup $H\wr
([n],\ov{[n]})$ (Lemma~\ref{structureofwreath}) and hence is simple.
Consideration of the projection $M\to \ov{([n],C_n)}$, which is onto
by definition of $x,y$, shows $J\subseteq M\cap H\wr
([n],\ov{[n]})$, since $\ov{[n]}$ is the minimal ideal of
$\ov{([n],C_n)}$.  We conclude $J=  M\cap H\wr ([n],\ov{[n]})$ and
in particular $y\in J$.

  According to
Lemma~\ref{structureofwreath} the map sending $(f,\ov 1)$ to $1f$
restricts to an isomorphism $\theta$ from the maximal subgroup of
$H\wr ([n],\ov{[n]})$ at $y$ to $H$.  We show that $\theta$ is still
surjective when restricted to the maximal subgroup $G_y$ of $M$ at
$y$.  In fact, we show that each element of $H$ is $\theta(z)$ for
some $z\in G_y$ which is a product of idempotents.  Graham's
theorem~\cite{Graham,qtheor} implies a simple semigroup is
generated by idempotents if and only if each element of the maximal
subgroup is a product of idempotents and so this will complete the
proof of Lemma~\ref{technicallemma2}.  Actually, the proof of
Theorem~\ref{Theorem2} can be made to work using only that each
element of $G_y$ is a product of idempotents.

The key observation is $x^jyx^{-j}$ is an idempotent for $1\leq
j\leq m-1$ and \[x^jyx^{-j}=(\ov e,a^j)(Y,\ov 1)(\ov e,a^{-j}) =
({}^{a^{j}}\!{Y},\ov {n-j+1)}).\]  Therefore, the element $z\in G_y$
defined by \[z=yx^jyx^{-j}y= (Y,\ov 1)({}^{a^{j}}\!{Y},\ov
{n-j+1})(Y,\ov 1) = (Y({}^{\ov {j+1}}\!{Y})({}^{\ov {n-j+1}}\!{Y}),\ov
1)\] is a product of idempotents and \[\theta(z)=1\left(Y({}^{\ov
{j+1}}\!{Y})({}^{\ov {n-j+1}}\!{Y})\right)=1Y\cdot (j+1)Y\cdot (n-j+1)Y =
h_{j+1}\] since $1Y=e=(n-j+1)Y$, where the latter equality holds
since $n\geq 2m-1$ implies $n-j+1\geq n-m+2\geq m+1$.  This
establishes each element of $H$ is $\theta(z)$ for some $z\in G_y$
which is a product of idempotents, finishing the proof of
Lemma~\ref{technicallemma2}.\qed

\section{The proof of Lemma~\ref{maintechnical}}\label{s:proofofmaintech}
The proof of Lemma~\ref{maintechnical} relies heavily on the wreath
product and forms the technical core of this paper.  The
construction is reminiscent of the one used in the proof of
Lemma~\ref{technicallemma2}. We shall find it convenient to use the
formulation of wreath products in terms of row monomial matrices.
Let $S$ be a semigroup.  Then $RM_n(S)$ denotes the monoid of all
$n\times n$ row monomial matrices with entries in $S$; in other
words it consists of all matrices over $S\cup \{0\}$ such that each
row has exactly one non-zero entry. The binary operation is usual
matrix multiplication. It is well known~\cite{Arbib,qtheor} $RM_n(S)
\cong S\wr ([n],T_n)$ where $T_n$ is the full transformation monoid
of degree $n$.  An element $(f,a)$ corresponds to the matrix $M$
with $M_{i,ia}=if$, $1\leq i\leq n$, and all other entries zero. In
particular, if $a$ is a constant map to $j$, then $M$ has all its
non-zero entries in column $j$.

From this viewpoint, an iterated wreath product $S\wr (B,T)\wr
(A,U)$ can be viewed as $|A|\times |A|$ block row monomial matrices
where the blocks are $|B|\times |B|$ row monomial matrices over $S$.
The term \emph{block entry} shall mean a matrix from $S\wr (B,T)$
while the term \emph{entry} shall always mean an element of the
semigroup $S$. Having dispensed with the preliminaries, we now turn
to the proof of Lemma~\ref{maintechnical}

Let $B$ be the set of $\L$-classes of $J$.  Denote by
$\mathsf{RLM}_J(M)$ the quotient of $M$ by the kernel of its action
on the right of $B$; note that $\mathsf{RLM}_J(M)$ contains all the
constant maps. Since the Sch\"utzenberger representation of $M$ on $I$ is
faithful, we can view $M$ as a monoid of $b\times b$ row monomial
matrices over $K$ where $b=|B|$. Moreover, the discussion in
Section~\ref{minimal} shows that an element $k$ of the maximal
subgroup $K$ at $\p(e)$
can be identified with the row monomial matrix having $k$ in every
entry of the first column.  For $x\in \ovFP H
X$, denote by $M_x$ the row monomial matrix associated to $\p(x)$.

Let $N=\ker \alpha$ and choose a set-theoretic section $\sigma:K\to
H$.  Then $H=N\sigma(K)$.  Denote by $M^{\sigma}_x$ the row monomial
matrix over $H$ obtained from $M_x$ by applying $\sigma$ entry-wise.
Let $n=|N|$ and let $m$ be a positive integer such that
$(M_{x_1}^{\sigma})^m$ is idempotent. Choose a prime
$p>\max\{m,n^b\}$ so that $\mathbb Z/p\mathbb Z\in \pv H$; such a
prime exists by our assumption on \pv H.
Denote by $C_p$ the cyclic group of order $p$ generated by the
permutation $(1\ 2\cdots p)$. Our monoid $M'$ will be a certain
submonoid of the iterated wreath product
\[W=H\wr (B,\mathsf{RLM}_J(M))\wr \ov{([p],C_p)}.\]  Observe that
$W\in \ov{\pv H}$.

We begin our construction of $M'$ by defining
\begin{equation*}
\til x_1 = \begin{bmatrix} 0 & M_{x_1}^{\sigma}& 0              &\cdots & 0     \\
                            0 & 0               &M_{x_1}^{\sigma}&0      &\cdots \\
                            0 & 0               &  0             &\ddots & 0     \\
                            0 & 0               & \cdots         & 0     & M_{x_1}^{\sigma}\\
                            M_{x_1}^{\sigma}& 0 &\cdots          &
                            0&0\end{bmatrix}.
\end{equation*}
In other words $\til x_1$ acts on the $[p]$ component by the cyclic
permutation $(1\ 2\ \cdots p)$ and each block entry of $\til x_1$
from $H\wr (B,\mathsf{RLM}_J(M))$ is $M_{x_1}^{\sigma}$. Set
$\ell=n^b$; so $p>\ell$ by choice of $p$. Let
$1=N_1,N_2\ldots,N_\ell$ be the distinct elements of $N^b$. We
identify $N^b$ with the group of diagonal $b\times b$ matrices over
$N$.  In particular, $N^b$ is a subgroup of $H\wr
(B,\mathsf{RLM}_J(M))$, as $M$ is a monoid.  In fact, there is a
natural onto homomorphism \[\ov{\alpha}:H\wr
(B,\mathsf{RLM}_J(M))\to K\wr (B,\mathsf{RLM}_J(M))\] induced by
$\alpha:H\to K$ and, moreover, it is straightforward to verify that
$\ov {\alpha}(x)=\ov {\alpha}(y)$ if and only if $x=N_jy$ some
$1\leq j\leq \ell$.  The map $\ov{\alpha}$ simply applies $\alpha$
entry-wise.

Next let us define, for $i=2,\ldots, n$, a $p\times p$ block row
monomial matrix by
\begin{equation*}
\til x_i = \begin{bmatrix} M_{x_i}^{\sigma}        & 0     &\cdots & 0     \\
                           N_2M_{x_i}^{\sigma}     & 0     &\cdots & 0 \\
                           \vdots                 & \vdots&\vdots & \vdots     \\
                           N_{\ell}M_{x_i}^{\sigma}&  0    &\cdots &   0   \\
                           M_{x_i}^{\sigma}        &  0    &\cdots &0 \\
                              \vdots               & \vdots&\vdots &\vdots\\
                           M_{x_i}^{\sigma}        &0      &\cdots &0   \end{bmatrix}
\end{equation*}
so $\til x_i$ has all its non-zero blocks in the first column.  The
$j^{th}$ block entry of the first column is $N_jM_{x_i}^{\sigma}$ if
$j\leq \ell$ and otherwise is $M_{x_i}^{\sigma}$. Then $\til
x_1,\ldots,\til x_n\in W$ and we have a map $X\to W$ given by
$x_i\mapsto \til x_i$.  Extend this to a continuous morphism
$\eta:\ovFP H X\to W$ and set $M'=\eta(\ovFP H X)$.  Our goal is to
show $M'$ is the desired monoid.  We begin by verifying that $\p$
factors through $\eta$.

\begin{Prop}\label{formnumber1}
Let $u\in \ovFP H X$.  Then each $U\in H\wr (B,\mathsf{RLM}_J(M))$
appearing as a block entry of $\eta(u)$ satisfies
$\ov{\alpha}(U)=M_u$.  As a consequence $\eta(u)=\eta(u')$ implies
$\p(u)=\p(u')$ and so $\p$ factors through $\eta$ as $\rho\eta$
where $\rho:M\to M'$ takes $\eta(u)$ to $\ov{\alpha}(U)$ where $U$
is any block entry of $\eta(u)$.
\end{Prop}
\begin{proof}
The second statement is immediate from the first.  We prove the
first statement for words $w\in X^*$ by a simple induction on
length, the case $|w|=0$ being trivial.  If $w=x_1u$, then the
definition of $\til x_1$ implies the block entries of $\eta(w)$ are
of the form $M_{x_1}^{\sigma}U$ where $U$ runs over the block
entries of $\eta(u)$ and the result follows by induction. The case
$w=x_iu$, $2\leq i\leq n$, is similar only the block entries of
$\eta(w)$ are now of the form $N_jM_{x_i}^{\sigma}U$ with $U$ the
block entry of $\eta(u)$ in the first row. If $u\in \ovFP H X$, then
since $X^*$ is dense, there exists a word $w\in X^*$ such that
$\eta(u)=\eta(w)$ and $M_u=M_w$.  The result now follows from the
case of words.
\end{proof}

 Our next goal is to show that if $w$ is a word whose
support contains some letter other than $x_1$, then each preimage of
$M_w$ under $\ov {\alpha}$ is a block entry of $\eta(w)$.  This will
be crucial in showing that the maximal subgroup of the minimal ideal
of $M'$ is isomorphic to $H$.  To effect this we shall need the following
lemma.

\begin{Lemma}\label{preimages}
Let $u,w\in \ovFP H X$ and suppose $W_1,\ldots, W_{\ell}$ are the
preimages of $M_w$ under $\ov {\alpha}$ and $U$ is a fixed preimage
of $M_u$ under $\ov {\alpha}$. Then $W_1U,\ldots, W_{\ell}U$,
respectively $UW_1,\ldots,UW_{\ell}$, are all the preimages of
$M_wM_u$, respectively $M_uM_w$, under $\ov {\alpha}$.
\end{Lemma}
\begin{proof}
The preimages of $M_w$ under $\ov{\alpha}$ are
$N_1M_w^{\sigma},\ldots,N_{\ell}M_w^{\sigma}$.  But as
$M_w^{\sigma}U$ is a preimage of $M_wM_u$, it follows
$\{N_1M_w^{\sigma}U,\ldots, N_{\ell}M_w^{\sigma}U\}$ is the complete
set of preimages of $M_wM_u$ under $\ov{\alpha}$.  For the preimages
of $M_uM_w$, note that $UN_1,\ldots,UN_{\ell}$ are the
$\ov{\alpha}$-preimages of $M_u$ so the previous case applies.
\end{proof}

Observe that if $w\in X^*$ and the support of $w$ is not contained
in $\{x_1\}$, then by  definition of $\til x_2,\ldots,\til x_n$,
the block entries of $\eta(w)$ form a single column, that is the
$\ov{([p],C_p)}$ component of $\eta(w)$ is a constant map.
We can now prove the aforementioned fact concerning preimages.

\begin{Prop}\label{allthere}
Let $w\in X^*$ have support not contained in $\{x_1\}$.  Then each
preimage of $M_w$ under $\ov {\alpha}$ appears as a block entry of
$\eta(w)$.
\end{Prop}
\begin{proof}
Let $S$ be the set of words in $X^*$ with support containing an
element outside of $\{x_1\}$.  We proceed by induction on $|w|$. If
$|w|=1$, then the proposition follows from the definition of $\til
x_2,\ldots,\til x_n$.

Suppose it is true for words in $S$ of length $n$ and let $w\in S$
have length $n+1$.  If the first
letter of $w\neq x_1$, then $w=ux_i$ with $u\in S$ some $i$; else $w=x_1u$
where $u\in S$.  In the case $w=x_1u$ the block entries of $\eta(w)$
are precisely the products of the form $M_{x_1}^{\sigma}U$ where $U$
runs over the block entries of $\eta(u)$.  By induction and
Lemma~\ref{preimages} it follows that the block entries of $\eta(w)$
are as required.  In the case $w=ux_i$, the block entries of
$\eta(u)$ are in a single column, say column $j$.  Let $V$ be the block
entry in row $j$ of $\til x_i$.  Then the block entries of $\eta(w)$
are all products of the form $UV$ where $U$ is a block entry of
$\eta(u)$. So again Lemma~\ref{preimages} yields each
$\ov{\alpha}$-preimage of $M_w$ is a block entry of $\eta(w)$.  This
completes the proof.
\end{proof}


\begin{Cor}\label{minimalidealguyallthere}
If $w\in I$, then the block entries of $\eta(w)$ are in a single
column and each preimage under $\ov{\alpha}$ of $M_w$ appears as a
block entry of $\eta(w)$.
\end{Cor}
\begin{proof}
Since $\ov{\pv H}$ contains the free semilattice $(P(X),\cup)$ it
follows that if $\{w_r\}$ is a sequence of words in $X^*$ converging to
$w$, then there exists $R$ such that for $r\geq R$ the word $w_r$
has support $X$.  Now there exists $s\geq R$ so that
$\eta(w)=\eta(w_s)$.  Since $w_s$ has full support, the corollary
then follows from Proposition~\ref{allthere} and the remark
preceding that proposition.
\end{proof}

By Corollary~\ref{minimalidealguyallthere} if $w\in I$, then the
$\ov{([p],C_p)}$ component of $\eta(w)$ is a constant map, that is
the block entries of $\eta(w)$ appear in a single column. Moreover,
Proposition~\ref{formnumber1} shows that each block entry of
$\eta(w)$ is a preimage of $M_w$ under $\ov{\alpha}$. But $M_w$,
being in the minimal ideal, has the shape of a constant map, i.e.\
it has only one non-zero column.  Hence $\eta(w)$ has all its
entries in a single column, that is, the $(B,\mathsf{RLM}_J(M))\wr
\ov{([p],C_p)}$ component of $\eta(w)$ is a constant map. Since
$\eta(I)$ is the minimal ideal $J'$ of $M'$
(Proposition~\ref{minontominimal}), we conclude $J'\subseteq M'\cap
H\wr (B\times [p],\ov{B\times [p]})$ and hence is simple by
Lemma~\ref{structureofwreath}.  Since simple semigroups form a
variety of finite semigroups, $M'\cap H\wr (B\times [p],\ov{B\times
  [p]})$ is simple. Therefore, we in fact have $J'=M'\cap  
H\wr (B\times [p],\ov{B\times [p]})$.  It remains to construct an
isomorphism $\theta:G_{\eta(e)}\to H$ such that
$\rho\theta\inv=\alpha$.

First note that since $ex_2^{\omega}=e$, it must be the case
$\eta(e)$ is a block matrix with each block entry in the first
column.  Also, the discussion in Section~\ref{minimal} indicates
$M_e$ is a matrix whose only non-zero column is the first column and
whose non-zero entries are comprised by the identity of $K$. Since
the block entries of $\eta(e)$ are preimages of $M_e$ under
$\ov{\alpha}$ (Proposition~\ref{formnumber1}), we deduce that all the
entries of $\eta(e)$ are in the
first column and belong to $N$. Lemma~\ref{structureofwreath} says
the map $\theta:H\wr (B\times [p],\ov{B\times [p]})\to H$ selecting the
$1,1$ entry is an isomorphism from the maximal subgroup at $\eta(e)$
of $H\wr (B\times [p],\ov{B\times [p]})$ to $H$.  In particular, the
$1,1$ entry of $\eta(e)$ is the identity of $H$.

We need to show that the restriction of $\theta$ to $G_{\eta(e)}$ is
onto and $\rho\theta\inv=\alpha$.  Let us prove the second assertion
assuming the first.  By Proposition~\ref{formnumber1} if $w\in \ovFP
H X$ maps under $\eta$ to $\theta\inv (h)$, then $\p(w)=\rho\eta(w)$
is obtained by choosing say the $1,1$ block entry of $\eta(w)$ and
applying $\alpha$ entry-wise. Since an element $k$ of $K=G_{\p(e)}$,
viewed as a row monomial matrix, has $k$ as each non-zero entry, it
follows the image of $\theta\inv(h)$ in $K$ is obtained by
evaluating $\alpha$ on the $1,1$ entry of $\eta(w)$.  But the $1,1$
entry of $\eta(w)$ is precisely $h$ by definition of $\theta$ so
$\rho\theta\inv (h)=\alpha(h)$, as required.

Thus we are left with proving $\theta$ is onto.  Since $\rho$ must
take $G_{\eta(e)}$ onto $K$ (Proposition~\ref{minontominimal}), it
follows from the discussion in the previous paragraph that we must
be able to obtain a preimage under $\alpha$ of each element of $K$
as the $1,1$ entry of some element of $G_{\eta(e)}$, that is,
$\alpha(\theta(G_{\eta(e)}))=K$. So it suffices to prove $\ker\alpha
=N$ is contained in the image of $\theta$.

Recall that $p$ was chosen so that $p>m$ where $(M_{x_1}^{\sigma})^m
= (M_{x_1}^{\sigma})^{\omega}$.  We can thus find a positive integer
$r$ so that $1\equiv rm\bmod p$.
Then \[\til x_1^{mr} = \begin{bmatrix} 0 & (M_{x_1}^{\sigma})^{\omega}& 0              &\cdots & 0     \\
                            0 & 0               &(M_{x_1}^{\sigma})^{\omega}&0      &\cdots \\
                            0 & 0               &  0             &\ddots & 0     \\
                            0 & 0               & \cdots         & 0     & (M_{x_1}^{\sigma})^{\omega}\\
                            (M_{x_1}^{\sigma})^{\omega}& 0 &\cdots          &
                            0&0\end{bmatrix}.\]
Set $C=\til x_1^{mr}$.  Then $C^j$ has the block form of the
permutation matrix corresponding to $(1\ 2\ \cdots p)^j$ and each
block entry of $C^j$ is $(M_{x_1}^{\sigma})^{\omega}$.  By
Corollary~\ref{minimalidealguyallthere} each preimage of $M_e$ under
$\overline{\alpha}$ appears as a block entry of $\eta(e)$.  Since
$\overline{\alpha}((M_{x_1}^{\sigma})^{\omega}) = M_{x_1}^{\omega}$
and $x_1^{\omega}e=e$, it follows from Lemma~\ref{preimages} that
the elements of the form $(M_{x_1}^{\sigma})^{\omega}U$, where $U$
runs over the block entries of $\eta(e)$, yield all the preimages of
$M_e$ under $\overline{\alpha}$ (with perhaps some repetition). Each
such matrix is the $1,1$ block entry of a product $C^jM_e$ for a
correctly chosen $j$. Now the $\ov{\alpha}$-preimages of $M_e$ are
the matrices whose first column has entries from $N$ and whose
remaining columns consist of zeroes.  Consequently any element of
$N$ can be the $1,1$ entry of an $\ov{\alpha}$-preimage of $M_e$ and
so every element of $N$ is the $1,1$ entry of some $C^j\eta(e)$.
Since $\eta(e)$ has the identity of $H$ in the $1,1$ entry,
$\eta(e)C^j\eta(e)$ is an element of $G_{\eta(e)}$ with the same
$1,1$ entry as $C^j\eta(e)$. Thus $\theta(G_{\eta(e)})$ contains $N$
as required. This completes the proof of Lemma~\ref{maintechnical},
thereby establishing Theorems~\ref{Theorem1} and~\ref{Theorem2}.\qed

\bibliographystyle{abbrv}
\bibliography{standard}

\def\malce{\mathbin{\hbox{$\bigcirc$\rlap{\kern-7.75pt\raise0,50pt\hbox{${\tt
  m}$}}}}}\def\cprime{$'$} \def\cprime{$'$} \def\cprime{$'$} \def\cprime{$'$}
  \def\cprime{$'$}
\begin{thebibliography}{10}

\bibitem{synthesis}
D.~Allen, Jr. and J.~Rhodes.
\newblock Synthesis of the classical and modern theory of finite semigroups.
\newblock {\em Advances in Math.}, 11(2):238--266, 1973.

\bibitem{Almeida:book}
J.~Almeida.
\newblock {\em Finite semigroups and universal algebra}, volume~3 of {\em
  Series in Algebra}.
\newblock World Scientific Publishing Co. Inc., River Edge, NJ, 1994.

\bibitem{Jorgesubgroup}
J.~Almeida.
\newblock Profinite groups associated with weakly primitive substitutions.
\newblock {\em Fundam. Prikl. Mat.}, 11(3):13--48, 2005.

\bibitem{AlmeidaVolkovfirst}
J.~Almeida and M.~V. Volkov.
\newblock Profinite identities for finite semigroups whose subgroups belong to
  a given pseudovariety.
\newblock {\em J. Algebra Appl.}, 2(2):137--163, 2003.

\bibitem{AlmeidaVolkov}
J.~Almeida and M.~V. Volkov.
\newblock Subword complexity of profinite words and subgroups of free profinite
  semigroups.
\newblock {\em Internat. J. Algebra Comput.}, 16(2):221--258, 2006.

\bibitem{CP}
A.~H. Clifford and G.~B. Preston.
\newblock {\em The algebraic theory of semigroups. {V}ol. {I}}.
\newblock Mathematical Surveys, No. 7. American Mathematical Society,
  Providence, R.I., 1961.

\bibitem{Graham}
R.~L. Graham.
\newblock On finite {$0$}-simple semigroups and graph theory.
\newblock {\em Math. Systems Theory}, 2:325--339, 1968.

\bibitem{Iwasawa}
K.~Iwasawa.
\newblock On solvable extensions of algebraic number fields.
\newblock {\em Ann. of Math (2)}, 58:548--572, 1953.

\bibitem{Arbib}
K.~Krohn, J.~Rhodes, and B.~Tilson.
\newblock {\em Algebraic theory of machines, languages, and semigroups}.
\newblock Edited by Michael A. Arbib. Academic Press, New York, 1968.
\newblock Chapters 1, 5--9.

\bibitem{RhodesStein}
J.~Rhodes and B.~Steinberg.
\newblock Profinite semigroups, varieties, expansions and the structure of
  relatively free profinite semigroups.
\newblock {\em Internat. J. Algebra Comput.}, 11(6):627--672, 2001.

\bibitem{projective}
J.~Rhodes and B.~Steinberg.
\newblock Closed subgroups of free profinite monoids are projective profinite
  groups.
\newblock {\em Bull. London Math. Soc.}, to appear.

\bibitem{qtheor}
J.~Rhodes and B.~Steinberg.
\newblock {\em The {$\mathfrak q$}-theory of finite semigroups}.
\newblock Springer, To appear.

\bibitem{RZbook}
L.~Ribes and P.~Zalesskii.
\newblock {\em Profinite groups}, volume~40 of {\em Ergebnisse der Mathematik
  und ihrer Grenzgebiete. 3. Folge. A}.
\newblock Springer-Verlag, Berlin, 2000.

\bibitem{Schutzmonomial}
M.-P. Sch{\"u}tzenberger.
\newblock Sur la repr\'esentation monomiale des demi-groupes.
\newblock {\em C. R. Acad. Sci. Paris}, 246:865--867, 1958.

\end{thebibliography}

\end{document}